\documentclass[draft,11pt,russian]{article}

\usepackage{amsfonts,amsmath,amsthm,amscd,amssymb,latexsym,cite,verbatim,texdraw,floatflt,caption2,pb-diagram}

    \usepackage[T2A]{fontenc} 
    \usepackage[cp1251]{inputenc}
    \usepackage[russian]{babel}

\theoremstyle{plain}
\newtheorem{theorem}{Теорема}[section]
\newtheorem{lemma}{Лемма}[section]

\newtheorem{corollary}{Следствие}[section]

\theoremstyle{definition}
\newtheorem{example}{Пример}[section]

\newcommand{\keywords}{\textbf{Ключевые слова. }}
\newcommand{\subjclass}{\textbf{MSC 2010. }}
\renewcommand{\abstract}{\textbf{Аннотация. }}

\def\Re{\mathop\mathrm{Re}}\,
\def\Im{\mathop\mathrm{Im}}\,

\numberwithin{equation}{section}

\begin{document}

\title{О моногенных отображениях кватернионной переменной\footnote{Работа выполнена при поддержке Министерства образования и науки Украины (проект ''Моногенные функции в банаховых алгебрах и краевые задачи анализа и математической физики'').}}

\author{Виталий С. Шпаковский и Татьяна С. Кузьменко}



\date{}

\maketitle

\begin{abstract}
В работе \cite{Shpakivskiy-Kuzmenko} рассмотрен класс,
 так называемых, $G$-моногенных (дифференцируемых по Гато) кватернионных отображений.
  В этой работе введены кватернионные $H$-моногенные (дифференцируемые по Хаусдорфу) отображения и установлена связь между $G$-моногенными и $H$-моногенными отображениями. Доказана эквивалентность разных определений $G$-моногенного отображения.
\end{abstract}

\subjclass{30G35, 57R35}

\keywords{Алгебра комплексных кватернионов, $G$-моногенные отображения, теорема Морера, $H$-моногенные отображения.}

\section{Введение}

Проблеме определения аналитической функции в ассоциативных (коммутативных или некоммутативных) алгебрах посвящено много работ  (см., например, \cite{Shpakivskiy-Kuzmenko} --- \cite{Dzagnidze}). В частности, в работах  \cite{Sinkov} --- \cite{Dzagnidze} указанная проблема рассматривается в алгебре кватернионов.

В то же время в кватернионном анализе
осталось незамеченным определение аналитической функции по Хаусдорфу ($H$-аналитической) \cite{Hausdorff}, не смотря на то, что в работах \cite{Ringleb}, \cite{Volov-1948} --- \cite{Rinehart} предпринимались некоторые попытки построения теории $H$-аналитических функций в общей ассоциативной алгебре.

Так, Ф.~Ринглеб в работе \cite{Ringleb} развивает теорию $H$-аналитических функций в произвольной
 конечномерной полупростой (т.~е., являющейся прямой суммой простых подалгебр) алгебре над полем действительных чисел $\mathbb{R}$. При этом он рассматривает функции, определенные и принимающие значения во всей алгебре.

Развивая идеи Хаусдорфа, С.~Воловельская в работе \cite{Volov-1948} определяет $H$-аналитические функции в области алгебры из некоторого класса конечномерных неполупростых алгебр над полем $\mathbb{R}$ и описывает общий вид таких функций.

В заметке М.~Дегтеревой \cite{Dehtereva} показано, что в коммутативной алгебре над $\mathbb{R}$ дифференцируемость по Хаусдорфу совпадает с дифференцируемостью по Шефферсу (см. \cite{Scheffers}).
В.~Портман \cite{Portman} определяет производную от $H$-аналитической функции в ассоциативных алгебрах над полем комплексных чисел $\mathbb{C}$ и исследует вопрос о ее соотношении с некоторыми другими определениями производной.

В работе Р.~Ринехарта и Дж.~Вилсона \cite{Rinehart} вводится класс функций, в некотором смысле дифференцируемых в любой ассоциативной алгебре над полем $\mathbb{R}$ или $\mathbb{C}$, и  изучается вопрос о соотношении между этими функциями и $H$-аналитическими функциями на различных классах алгебр.

В нашей работе \cite{Shpakivskiy-Kuzmenko} в алгебре комплексных кватернионов был
определен класс $G$-моногенных (дифференцируемых по Гато) отображений.
Там же установлено конструктивное описание всех отображений из этого класса с помощью четырех аналитических функций комплексной переменной. В работе \cite{Shp-Kuzm-intteor} для $G$-моногенных отображений доказаны аналоги интегральной теоремы Коши для криволинейного и поверхностного интеграла и интегральной формулы Коши. Кроме того, в статье \cite{Kuzm-series}
получены разложения $G$-моногенных отображений в ряды Тейлора и Лорана, а также проведено классификацию особых точек рассматриваемых отображений.

В этой работе мы вводим класс $H$-моногенных (дифференцируемых по Хаусдорфу) отображений в алгебре комплексных кватернионов и устанавливаем связь между $G$-моногенными и $H$-моногенными отображениями. Кроме того, доказывается теорема об эквивалентности разных определений $G$-моногенного отображения.

\section{Алгебра комплексных кватернионов}

Пусть $\mathbb{H(C)}$ --- алгебра кватернионов над полем комплексных чисел $\mathbb{C}$, базис которой состоит из единицы алгебры $1$ и элементов $I,J,K$, для которых выполняются правила умножения:
$$I^2=J^2=K^2=-1,$$
$$IJ=-JI=K,\qquad JK=-KJ=I,\qquad KI=-IK=J.$$

В алгебре $\mathbb{H(C)}$ существует другой базис $\{e_1,e_2,e_3,e_4\}$:
$$e_1=\frac{1}{2}(1+iI), \quad e_2=\frac{1}{2}(1-iI), \quad
e_3=\frac{1}{2}(iJ-K), \quad e_4=\frac{1}{2}(iJ+K),$$
где $i$ --- мнимая комплексная единица. Таблица умножения в новом базисе принимает следующий вид (см., например, \cite{Cartan})
$$\begin{tabular}{c||c|c|c|c|}
$\cdot$ & $e_1$ & $e_2$ & $e_3$ & $e_4$\\
\hline
\hline
$e_1$ & $e_1$ & $0$ & $e_3$ & $0$\\
\hline
$e_2$ & $0$ & $e_2$ & $0$ & $e_4$\\
\hline
$e_3$ & $0$ & $e_3$ & $0$ & $e_1$\\
\hline
$e_4$ & $e_4$ & $0$ & $e_2$ & $0$\\
\hline
\end{tabular}\,\,,$$
при этом единица алгебры представляется в виде $1=e_1+e_2$. Очевидно, что коммутативная подалгебра с базисом $\{e_1,e_2\}$ является алгеброй бикомплексных чисел или алгеброй коммутативних кватернионов Сегре \cite{Segre}.

Напомним (см., например, \cite[c. 64]{van_der_Varden}), что  подмножество
$\mathcal{I}\subset\mathbb{H(C)}$
называется \textit{левым} (или \textit{правым}) \textit{идеалом}, если из
 условия $x\in\mathcal{I}$
следует $yx\in\mathcal{I}$\, (или $xy\in\mathcal{I}$) для любого $y\in\mathbb{H(C)}$.

Теперь отметим, что алгебра $\mathbb{H(C)}$ содержит два правых максимальных идеала
$$\mathcal{I}_1:=\{\lambda_2e_2+\lambda_4e_4:\lambda_2,\lambda_4\in\mathbb{C}\},\qquad
\mathcal{I}_2:=\{\lambda_1e_1+\lambda_3e_3:\lambda_1,\lambda_3\in\mathbb{C}\}$$
и два левых максимальных идеала
$$\mathcal{\widehat{I}}_1:=\{\lambda_2e_2+\lambda_3e_3:\lambda_2,\lambda_3\in
\mathbb{C}\},\qquad
\mathcal{\widehat{I}}_2:=\{\lambda_1e_1+\lambda_4e_4:\lambda_1,\lambda_4\in\mathbb{C}\}.$$

Следствием очевидных равенств $$\mathcal{I}_1\cap\mathcal{I}_2=\widehat{\mathcal{I}}_1\cap\widehat{\mathcal{I}}_2=
0,\qquad \mathcal{I}_1\cup\mathcal{I}_2=\widehat{\mathcal{I}}_1\cup
\widehat{\mathcal{I}}_2=\mathbb{H(C)}
$$
является разложение в прямую сумму:
$$\mathbb{H(C)}=\mathcal{I}_1\oplus\mathcal{I}_2=
\widehat{\mathcal{I}}_1\oplus\widehat{\mathcal{I}}_2.
$$

Определим линейные функционалы $f_1:\mathbb{H(C)\rightarrow\mathbb{C}}$ и $f_2:\mathbb{H(C)\rightarrow\mathbb{C}}$, полагая
$$f_1(e_1)=f_1(e_3)=1, \qquad f_1(e_2)=f_1(e_4)=0,$$
$$f_2(e_2)=f_2(e_4)=1, \qquad f_2(e_1)=f_2(e_3)=0,$$
при этом очевидно $f_1(\mathcal{I}_1)=f_2(\mathcal{I}_2)=0$.

Определим также линейные функционалы $\widehat{f}_1:\mathbb{H(C)\rightarrow\mathbb{C}}$ и $\widehat{f}_2:\mathbb{H(C)\rightarrow\mathbb{C}}$, полагая
$$\widehat{f}_1(e_1)=\widehat{f}_1(e_4)=1, \qquad \widehat{f}_1(e_2)=\widehat{f}_1(e_3)=0,$$
$$\widehat{f}_2(e_2)=\widehat{f}_2(e_3)=1, \qquad \widehat{f}_2(e_1)=\widehat{f}_2(e_4)=0,$$
для которых очевидно $\widehat{f}_1(\widehat{\mathcal{I}}_1)=
\widehat{f}_2(\widehat{\mathcal{I}}_2)=0$.

Отметим, что указанные функционалы являются непрерывными и в некотором смысле мультипликативными (см. \cite{Shpakivskiy-Kuzmenko}).

\section{$G$-моногенные отображения}

Пусть
\begin{equation}\label{vectors-i}
i_1=1, \quad i_2=a_1e_1+a_2e_2, \quad i_3=b_1e_1+b_2e_2
\end{equation}
при $a_k,b_k\in\mathbb{C},\,k=1,2$ --- тройка линейно независимых векторов над полем  $\mathbb{R}$. Это означает, что равенство
$$\alpha_1i_1+\alpha_2i_2+\alpha_3i_3=0, \qquad
\alpha_1,\alpha_2,\alpha_3\in\mathbb{R}$$ выполняется тогда и только тогда, когда
 $\alpha_1=\alpha_2=\alpha_3=0$.

 Выделим в алгебре $\mathbb{H(C)}$ линейную оболочку $E_3:=\{\zeta=xi_1+yi_2+zi_3:x,y,z\in\mathbb{R}\}$ над полем $\mathbb{R}$, порожденную векторами $i_1,i_2,i_3$. Введем обозначения
 $$\xi_1:=f_1(\zeta)=\widehat{f}_1(\zeta)=x+ya_1+zb_1,$$
$$\xi_2:=
f_2(\zeta)=\widehat{f}_2(\zeta)=x+ya_2+zb_2.$$
Теперь элемент $\zeta\in E_3$ может быть представлен в виде $\zeta=\xi_1e_1+\xi_2e_2$.

 Множеству $S\subset\mathbb{R}^3$ поставим в соответствие множество $S_\zeta:=
\{\zeta=xi_1+yi_2+zi_3:(x,y,z)\in S\}$ в $E_3$.

Пусть $\Omega$ --- область в $\mathbb{R}^3$. В работе \cite{Shpakivskiy-Kuzmenko} предложено следующее определение.

Непрерывное отображение $\Phi:\Omega_\zeta\rightarrow\mathbb{H(C)}$ (или $\widehat{\Phi}:\Omega_\zeta\rightarrow\mathbb{H(C)}$) называется \emph{право-$G$-моногенным}
\big(или  \emph{лево-$G$-моногенным}\big) в области
$\Omega_\zeta\subset E_3$,
если $\Phi$ \big(или  $\widehat{\Phi}$\big) дифференцируемо по Гато в каждой точке этой области, т.~е., если для каждого $\zeta\in
 \Omega_\zeta$ существует элемент $\Phi'(\zeta)$ \big(или $\widehat{\Phi}'(\zeta)$\big) алгебры $\mathbb{H(C)}$ такой, что выполняется равенство
\begin{equation}\label{ozn-r-G-monog}
\lim\limits_{\varepsilon\rightarrow 0+0}\Big(\Phi(\zeta+\varepsilon h)-\Phi(\zeta)\Big)\varepsilon^{-1}= h\Phi'(\zeta)\quad\forall\,h\in E_3
\end{equation}
$$\Biggr(\mbox{или}\,\, \lim\limits_{\varepsilon\rightarrow 0+0}
\left(\widehat{\Phi}(\zeta+\varepsilon h)-\widehat{\Phi}(\zeta)\right)
\varepsilon^{-1}= \widehat{\Phi}'(\zeta)h\quad\forall\,h\in E_3\Biggr).$$
При этом $\Phi'(\zeta)$ называется \emph{правой производной Гато} отображения $\Phi$, а $\widehat{\Phi}'(\zeta)$ --- \emph{левой производной Гато} отображения $\widehat{\Phi}$ в точке $\zeta$\,.

Рассмотрим разложение отображения $\Phi:\Omega_\zeta\rightarrow\mathbb{H(C)}$ по базису $\{e_1,e_2,e_3,e_4\}$:
\begin{equation}\label{Phi-U}
\Phi(\zeta)=\sum\limits_{k=1}^4U_k(x,y,z)e_k.
\end{equation}
В предположении, что функции $U_k:\Omega\rightarrow\mathbb{C}$ являются $\mathbb{R}$-дифференцируемыми в области $\Omega$, т. е. во всех точках $(x,y,z)\in\Omega$ выполняются соотношения
$$U_k(x+\Delta x,y+\Delta y,z+\Delta z)-U_k(x,y,z)=$$
$$=\frac{\partial U_k}{\partial x}\Delta x+\frac{\partial U_k}{\partial y}\Delta y+\frac{\partial U_k}{\partial z}\Delta z+o\left(\sqrt{(\Delta x)^2+(\Delta y)^2+(\Delta z)^2}\right),$$
$$(\Delta x)^2+(\Delta y)^2+(\Delta z)^2 \rightarrow 0,$$
в теореме 1 из \cite{Shpakivskiy-Kuzmenko} установлены необходимые и достаточные условия право-$G$-моногенности отображения $\Phi$ (или лево-$G$-моногенности отображения $\widehat{\Phi}$) (аналоги условий Коши--Римана), которые всюду в области $\Omega_\zeta$ в свернутом виде выражаются равенствами

\begin{equation}\label{umova-r-K-R}
\frac{\partial \Phi}{\partial y}=i_2\frac{\partial \Phi}{\partial x}\,, \qquad
 \frac{\partial \Phi}{\partial z}=i_3\frac{\partial \Phi}{\partial x}
\end{equation}
\begin{equation}\label{umova-l-K-R}
\left(\mbox{или }\quad \frac{\partial \widehat{\Phi}}{\partial y}=\frac{\partial \widehat{\Phi}}{\partial x}
i_2\,, \qquad \frac{\partial \widehat{\Phi}}{\partial z}=\frac{\partial
\widehat{\Phi}}{\partial x}i_3\,\right).
\end{equation}

Из леммы 2 работы \cite{Shpakivskiy-Kuzmenko} вытекает, что точки $(x,y,z)\in\mathbb{R}^3$, соответствующие необратимым элементам $\zeta=xi_1+yi_2+zi_3$, лежат на прямых
$$L^1: x+y\Re\,a_1+z\Re\,b_1=0,
\qquad y\Im\,a_1+z\Im\,b_1=0,$$
$$L^2: x+y\Re\,a_2+z\Re\,b_2=0,
\qquad y\Im\,a_2+z\Im\,b_2=0$$
в трехмерном пространстве $\mathbb{R}^3$.

Обозначим через $f_k(E_3)$ при $k=1,2$, --- образ множества $E_3$ при отображении $f_k$. Отметим, что существенным для дальнейшего изложения является предположение $f_1(E_3)=f_2(E_3)=\mathbb{C}$.   Очевидно, что оно имеет место тогда и только тогда, когда хотя бы одно из чисел в каждой из пар $(a_1,b_1)$, $(a_2,b_2)$ принадлежит $\mathbb{C}\setminus\mathbb{R}$.

Пусть $D_1$ и $D_2$ --- области в $\mathbb{C}$, на которые область $\Omega_\zeta$ отображается соответственно функционалами $f_1$ и $f_2$.

В теореме 5 из \cite{Shpakivskiy-Kuzmenko} описаны все право-$G$-моногенные отображения, определенные в области $\Omega_\zeta$ и принимающие значения в алгебре $\mathbb{H(C)}$, с помощью аналитических функций комплексной переменной. А именно, если область $\Omega\subset
\mathbb{R}^{3}$ выпукла в направлении прямых $L^1$, $L^2$ и $f_1(E_3)=f_2(E_3)=\mathbb{C}$, то
 каждое право-$G$-моногенное отображение $\Phi:\Omega_\zeta\rightarrow\mathbb{H(C)}$ представляется в виде
\begin{equation}\label{Phi-r-rozklad}
\Phi(\zeta)=F_1(\xi_1)e_1+F_2(\xi_2)e_2+F_3(\xi_1)e_3+F_4(\xi_2)e_4
\end{equation}
$$\forall\,\zeta=xi_1+yi_2+zi_3\in\Omega_\zeta\,,$$ где
 $F_1, F_3$ --- некоторые аналитические в области $D_1$ функции переменной $\xi_1=x+ya_1+zb_1$, а
  $F_2, F_4$ --- некоторые аналитические в области $D_2$ функции переменной $\xi_2=x+ya_2+zb_2$.

При таких же предположениях, каждое лево-$G$-моногенное отображение
$\widehat{\Phi}:\Omega_\zeta\rightarrow\mathbb{H(C)}$ представляется в виде
\begin{equation}\label{Phi-l-rozklad}
\widehat{\Phi}(\zeta)=\widehat{F}_1(\xi_1)e_1+\widehat{F}_2(\xi_2)e_2+\widehat{F}_3(\xi_2)e_3+
\widehat{F}_4(\xi_1)e_4,
\end{equation}
где
 $\widehat{F}_1, \widehat{F}_4$ --- некоторые аналитические в области $D_1$ функции переменной $\xi_1=x+ya_1+zb_1$, а
  $\widehat{F}_2, \widehat{F}_3$ --- некоторые аналитические в области $D_2$ функции переменной $\xi_2=x+ya_2+zb_2$.

Отметим, что производная Гато право-$G$-моногенного отображения $\Phi(\zeta)$ (или лево-$G$-моногенного отображения $\widehat{\Phi}(\zeta)$) вычисляется по формуле
$$\Phi'(\zeta)=F_1'(\xi_1)e_1+F_2'(\xi_2)e_2+F_3'(\xi_1)e_3+F_4'(\xi_2)e_4$$
$$\left(\mbox{или \quad}\widehat{\Phi}'(\zeta)=\widehat{F}_1'(\xi_1)e_1+\widehat{F}_2'(\xi_2)e_2+\widehat{F}_3'(\xi_2)e_3+
\widehat{F}_4'(\xi_1)e_4\right).$$

Используя представления (\ref{Phi-r-rozklad}), (\ref{Phi-l-rozklad}), в работе \cite{Kuzm-series} получены разложения $G$-моногенных отображений в ряды Тейлора. Если $f_1(E_3)=f_2(E_3)=\mathbb{C}$ и $\zeta_0:=x_0i_1+y_0i_2+z_0i_3\in\Omega_\zeta$, то каждое право-$G$-моногенное отображение $\Phi:\Omega_\zeta\rightarrow\mathbb{H(C)}$ представляется в виде суммы сходящегося степенного ряда
\begin{equation}\label{r-Taylor}
\Phi(\zeta)=\sum_{n=0}^{\infty}(\zeta-\zeta_0)^{n}\,p_n,\qquad p_n\in\mathbb{H(C)},
\end{equation}
а каждое лево-$G$-моногенное отображение $\widehat{\Phi}:\Omega_\zeta\rightarrow\mathbb{H(C)}$ --- в виде суммы сходящегося степенного ряда:
\begin{equation}\label{l-Taylor}
\widehat{\Phi}(\zeta)=\sum_{n=0}^{\infty}\widehat{p}_n\,(\zeta-\zeta_0)^{n},\qquad \widehat{p}_n\in\mathbb{H(C)}.
\end{equation}

\section{Теорема Морера}

Рассмотрим алгебру $\mathbb{\widetilde{H}(R)}$ с базисом $\{e_k,ie_k\}_{k=1}^4$ над полем действительных чисел $\mathbb{R}$, которая изоморфна алгебре $\mathbb{H(C)}$ над полем комплексных чисел $\mathbb{C}$. Очевидно, что в алгебре $\mathbb{\widetilde{H}(R)}$ существует базис $\{i_k\}_{k=1}^8$, где векторы $i_1,i_2,i_3$ те же, что и в соотношениях (\ref{vectors-i}).

Для элемента $a:=\sum\limits_{k=1}^{8}a_ki_k,\,a_k\in\mathbb{R}$ определим евклидову норму
$$\|a\|:=\sqrt{\sum\limits_{k=1}^8a_k^2}\,.$$
Соответственно, $\|\zeta\|=\sqrt{x^2+y^2+z^2}$ и $\|i_1\|=\|i_2\|=\|i_3\|=1.$

В силу теоремы об эквивалентности норм, для произвольного элемента
$b:=\sum\limits_{k=1}^4(b_{1k}+ib_{2k})e_k,\,b_{1k},b_{2k}\in\mathbb{R}$, выполняются неравенства
\begin{equation}\label{norma-ineq}
|b_{1k}+ib_{2k}|\leq\sqrt{\sum\limits_{k=1}^{4}(b_{1k}^2+b_{2k}^2)}\leq c\|b\|,
\end{equation}
где $c$ --- положительная постоянная, не зависящая от $b$.

Пусть $\gamma$ --- жорданова спрямляемая кривая в $\mathbb{R}^{3}$. Для непрерывной функции $\Psi:\gamma_{\zeta}\rightarrow
\mathbb{H(C)}$ вида
\begin{equation}\label{Phi-form}
\Psi(\zeta)=\sum\limits_{k=1}^{4}U_k(x,y,z)e_k+iV_k(x,y,z)e_k,
\end{equation}
где $(x,y,z)\in\gamma$ и
 $U_k:\gamma\rightarrow\mathbb{R}$,
$V_k:\gamma\rightarrow\mathbb{R}$,
определим интегралы по жордановой спрямляемой кривой
 $\gamma_\zeta$ равенствами

$$\int\limits_{\gamma_\zeta}{d\zeta\Psi(\zeta)}:=\sum\limits_{k=1}^{4}{e_k\int
\limits_{\gamma}U_k(x,y,z)dx}+\sum\limits_{k=1}^{4}{i_2e_k\int\limits_{\gamma}
U_k(x,y,z)dy}+$$
$$+\sum\limits_{k=1}^{4}{i_3e_k\int\limits_{\gamma}U_k(x,y,z)dz}+i\sum
\limits_{k=1}^{4}{e_k\int\limits_{\gamma}V_k(x,y,z)dx}+$$
$$+i\sum\limits_{k=1}^{4}{i_2e_k\int\limits_{\gamma}V_k(x,y,z)dy}+i\sum
\limits_{k=1}^{4}{i_3e_k\int\limits_{\gamma}V_k(x,y,z)dz}$$

\noindent и

$$\int\limits_{\gamma_\zeta}{\Psi(\zeta)d\zeta}:=\sum\limits_{k=1}^{4}
{e_k\int\limits_{\gamma}U_k(x,y,z)dx}+\sum\limits_{k=1}^{4}{e_ki_2\int
\limits_{\gamma}U_k(x,y,z)dy}+$$
$$+\sum\limits_{k=1}^{4}{e_ki_3\int\limits_{\gamma}U_k(x,y,z)dz}+i\sum
\limits_{k=1}^{4}{e_k\int\limits_{\gamma}V_k(x,y,z)dx}+$$
$$+i\sum\limits_{k=1}^{4}{e_ki_2\int\limits_{\gamma}V_k(x,y,z)dy}+i\sum
\limits_{k=1}^{4}{e_ki_3\int\limits_{\gamma}V_k(x,y,z)dz},$$
где $d\zeta:=dx+i_2dy+i_3dz$.

\begin{lemma}\label{Lema_do_Morera} Если $\gamma$ --- замкнутая жорданова спрямляемая кривая в $\mathbb{R}^3$ и функция $\Psi:\gamma_\zeta\rightarrow\mathbb{H(C)}$ непрерывна, то
\begin{equation}\label{lema_do_Morera-1}
\Biggr\|\int\limits_{\gamma_\zeta}d\zeta\,\Psi(\zeta)\Biggr\|\leq c\int\limits_{\gamma_\zeta}\|\Psi(\zeta)\|\|d\zeta\|
\end{equation}
\noindent и
\begin{equation}\label{lema_do_Morera-2}
\Biggl\|\int\limits_{\gamma_\zeta}\Psi(\zeta)\,d\zeta\Biggr\|\leq c\int\limits_{\gamma_\zeta}\|\Psi(\zeta)\|\|d\zeta\|,
\end{equation}
где $c$ --- абсолютная положительная постоянная.
\end{lemma}

\begin{proof}
Используя представление функции $\Psi$ в виде (\ref{Phi-form}), получаем оценку
$$\Biggl\|\int\limits_{\gamma_\zeta}d\zeta\Psi(\zeta)\Biggr\|\leq \sum\limits_{k=1}^4\|i_1e_k\|\int\limits_{\gamma}|U_k(x,y,z)+iV_k(x,y,z)|\,dx+$$
$$+\sum\limits_{k=1}^4\|i_2e_k\|\int\limits_{\gamma}|U_k(x,y,z)+iV_k(x,y,z)|\,dy+$$
$$+\sum\limits_{k=1}^4\|i_3e_k\|\int\limits_{\gamma}|U_k(x,y,z)+iV_k(x,y,z)|\,dz.$$
Принимая во внимание неравенство (\ref{norma-ineq}) при $b=\Psi(\zeta)$ и неравенства $\|i_se_k\|\leq c_s,\,\,s=1,2,3$, где $c_s$ --- абсолютные положительные постоянные, получаем оценку (\ref{lema_do_Morera-1}). Аналогично устанавливается оценка (\ref{lema_do_Morera-2}). Лемма доказана.
\end{proof}

Под треугольником $\Delta$ будем понимать плоскую фигуру ограниченную тремя отрезками, соединяющими три его вершины. Через $\partial\Delta$ обозначим границу треугольника $\Delta$ в относительной топологии его плоскости.

Используя лемму \ref{Lema_do_Morera} для отображений, принимающих значения в алгебре $\mathbb{H(C)}$, по стандартной схеме доказывается следующий аналог теоремы Морера.

\begin{theorem}\label{teor-Morera} Пусть $f_1(E_3)=f_2(E_3)=\mathbb{C}$. Если отображение $\Phi:\Omega_\zeta\rightarrow\mathbb{H(C)}$ $(\mbox{или}\,\,\,\widehat{\Phi}:\Omega_\zeta\rightarrow\mathbb{H(C)})$ непрерывно в области $\Omega_\zeta$ и удовлетворяет равенству
\begin{equation}\label{int-morera-r}
\int\limits_{\partial\Delta_\zeta}d\zeta\,\Phi(\zeta)=0
\end{equation}
\begin{equation}\label{int-morera-l}
\left(\mbox{или  } \int\limits_{\partial\Delta_\zeta}\widehat{\Phi}(\zeta)\,d\zeta=0\right)
\end{equation}
для каждого треугольника $\Delta_\zeta$ такого, что замыкание $\overline{\Delta_\zeta}\subset\Omega_\zeta$, то отображение $\Phi$ право-$G$-моногенное \big($\mbox{или  }\,\,\,\widehat{\Phi}$ --- лево-$G$-моногенное\big) в области $\Omega_\zeta$.
\end{theorem}

\section{$H-$моногенные отображения}

Ф.~Хаусдорф \cite{Hausdorff} предложил определение аналитической функции в любой ассоциативной (коммутативной или некоммутативной) алгебре $\mathbb{A}$ над полем $\mathbb{C}$ с единицей, которое может быть сформулировано следующим образом.

Гиперкомплексная функция
\begin{equation}\label{dF-varphi-}
f(\eta)=\sum\limits_{k=1}^{n}f_k(\eta_1,\ldots,\eta_n)e_k\,,
\end{equation} где $e_k$ --- базисные элементы алгебры $\mathbb{A}$, называется \emph{$H$-аналитической функцией} переменной $\eta:=\sum\limits_{k=1}^{n}\,\eta_ke_k$, если компоненты $f_k$ из разложения (\ref{dF-varphi-}) являются аналитическими функциями комплексных переменных $\eta_1,\ldots,\eta_n$ и дифференциал
\begin{equation}\label{dF-varphi+}
df:=\sum\limits_{k=1}^ndf_k(\eta_1,\ldots,\eta_n)e_k=\sum\limits_{j,k=1}^{n}\frac{\partial f_k}{\partial\eta_j}\,d\eta_j\,e_k
\end{equation}
является линейным однородным полиномом дифференциала $d\eta:=\sum\limits_{k=1}^{n}d\eta_k\,e_k$, т.~е.
\begin{equation}\label{dF-varphi--}
df=\sum\limits_{s=1}^{n^2}A_s\,d\eta\,B_s\,,
\end{equation}
где $A_s$ и $B_s$ --- некоторые $\mathbb{A}$--значные функции.

При этом значение $f'(\eta):=\sum\limits_{s=1}^{n^2}A_s\,B_s$ называют \emph{производной Хаусдорфа} функции $f(\eta)$.

Отметим, что в работе \cite{Ringleb} при определении $H$-аналитической функции в ассоциативной алгебре над полем $\mathbb{R}$, предполагается аналитичность
 действительнозначных компонент $f_k$ из разложения (\ref{dF-varphi-}), а в работе
\cite{Rinehart} рассматриваются ассоциативные алгебры над полями $\mathbb{R}$ или $\mathbb{C}$ и предполагается лишь существование частных производных $\frac{\partial f_k}{\partial\eta_j}$ при всех $j,k=1,2,\ldots,n$.

Подчеркнем, что свойство $H$-аналитичноcти функции не зависит от выбора базиса алгебры. Кроме того, если функции $f(\eta)$ и $g(\eta)$ $H$-аналитические, то функции $f(\eta)+g(\eta)$ и $f(\eta)\cdot g(\eta)$ также $H$-аналитические, при этом $d(f+g)=df+dg$ и $d(f\cdot g)=df\cdot g+f\cdot dg$\, (см. \cite{Ringleb,Portman}).

Теперь реализуем подход Хаусдорфа к отображениям переменной $\zeta=xi_1+yi_2+zi_3$.

Непрерывное отображение $\Phi:\Omega_\zeta\rightarrow\mathbb{H(C)}$ вида (\ref{Phi-U}) будем называть \emph{$H$-моногенным}
в области $\Omega_\zeta\subset E_3$,
если $\Phi$ дифференцируемо по Хаусдорфу в каждой точке $\zeta\in\Omega_\zeta$, т.~е. если компоненты отображения (\ref{Phi-U}) имеют частные производные первого порядка по переменным $x,y,z$, и формальный дифференциал отображения
\begin{equation}\label{form-dif}
d\Phi:=\sum\limits_{k=1}^4\left(\frac{\partial U_k}{\partial x}dx+\frac{\partial U_k}{\partial y}dy+\frac{\partial U_k}{\partial z}dz\right)e_k
\end{equation}
является линейным однородным полиномом от дифференциала $d\zeta=dx+i_2dy+i_3dz$, т.~е.
\begin{equation}\label{--riv-1}
d\Phi=\sum\limits_{s=1}^{16}A_s\,d\zeta\,B_s\,,
\end{equation}
где $A_s,\,B_s$ --- некоторые $\mathbb{H(C)}$--значные функции.

Отметим, что если частные производные первого порядка функций $U_k$ при $k=1,2,3,4$ существуют и непрерывны, то формальный дифференциал (\ref{form-dif}) будет полным дифференциалом отображения $\Phi$, т. е. является главной частью приращения этого отображения.

Как и выше, $\Phi'_H(\zeta):=\sum\limits_{s=1}^{16}A_sB_s$ назовем \emph{производной Хаусдорфа} отображения $\Phi(\zeta)$.

Покажем, что определение производной $\Phi'_H$ является корректным.

\begin{theorem}
Если отображение $\Phi:\Omega_\zeta\rightarrow\mathbb{H(C)}$ является $H$-моногенным в области $\Omega_\zeta$\,, то его производная $\Phi'_H$ существует и не зависит от выбора функций $A_s,\,B_s$ в равенстве \em(\ref{--riv-1})\em, при этом
$$\Phi'_H(\zeta)=\frac{\partial\Phi}{\partial x}.$$
\end{theorem}

\begin{proof}
Вследствие $H$-моногенности отображения $\Phi$ выполняется равенство
\begin{equation}\label{dif-teor}
\sum\limits_{s=1}^{16}A_sd\zeta B_s=\sum\limits_{k=1}^4\left(\frac{\partial U_k}{\partial x}dx+\frac{\partial U_k}{\partial y}dy+\frac{\partial U_k}{\partial z}dz\right)e_k\,.
\end{equation}

Пусть
$$A_s=a_{s1}e_1+a_{s2}e_2+a_{s3}e_3+a_{s4}e_4\,,$$
\vspace{-8mm}
\begin{equation}\label{A,B}
\end{equation}
\vspace{-8mm}
$$B_s=b_{s1}e_1+b_{s2}e_2+b_{s3}e_3+b_{s4}e_4$$
для $s=1,2,\ldots, 16.$
Учитывая равенство $d\zeta=(dx+a_1dy+b_1dz)e_1+(dx+a_2dy+b_2dz)e_2$ и (\ref{A,B}), получаем:
$$A_sd\zeta B_s=(a_{s1}e_1+a_{s2}e_2+a_{s3}e_3+a_{s4}e_4)\Big((dx+a_1dy+b_1dz)e_1+$$
$$+(dx+a_2dy+b_2dz)e_2\Big)(b_{s1}e_1+b_{s2}e_2+b_{s3}e_3+b_{s4}e_4)=$$
$$=\Big(a_{s1}b_{s1}(dx+a_1dy+b_1dz)+a_{s3}b_{s4}(dx+a_2dy+b_2dz)\Big)e_1+$$
$$+\Big(a_{s2}b_{s2}(dx+a_2dy+b_2dz)+a_{s4}b_{s3}(dx+a_1dy+b_1dz)\Big)e_2+$$
$$+\Big(a_{s1}b_{s3}(dx+a_1dy+b_1dz)+a_{s3}b_{s2}(dx+a_2dy+b_2dz)\Big)e_3+$$
\begin{equation}\label{A-dzeta-B}
+\Big(a_{s2}b_{s4}(dx+a_2dy+b_2dz)+a_{s4}b_{s1}(dx+a_1dy+b_1dz)\Big)e_4\,.
\end{equation}
Следствием равенств (\ref{dif-teor}) и (\ref{A-dzeta-B}) являются соотношения
$$\frac{\partial U_1}{\partial x}=\sum\limits_{s=1}^{16}a_{s1}b_{s1}+a_{s3}b_{s4}\,,\qquad \frac{\partial U_2}{\partial x}=\sum\limits_{s=1}^{16}a_{s2}b_{s2}+a_{s4}b_{s3}\,,$$
\vspace{-7mm}
\begin{equation}\label{pohidni}
\end{equation}
\vspace{-7mm}
$$\frac{\partial U_3}{\partial x}=\sum\limits_{s=1}^{16}a_{s1}b_{s3}+a_{s3}b_{s2}\,,\qquad \frac{\partial U_4}{\partial x}=\sum\limits_{s=1}^{16}a_{s2}b_{s4}+a_{s4}b_{s1}\,.$$

С учетом равенств (\ref{A,B}), имеем
$$\Phi'_H(\zeta):=\sum\limits_{s=1}^{16}A_s B_s=\sum\limits_{s=1}^{16}\Big((a_{s1}b_{s1}+a_{s3}b_{s4})e_1+$$
$$+(a_{s2}b_{s2}+a_{s4}b_{s3})e_2+(a_{s1}b_{s3}+a_{s3}b_{s2})e_3+(a_{s2}b_{s4}+a_{s4}b_{s1})e_4\Big),$$
откуда, принимая во внимание соотношения (\ref{pohidni}), получаем
$$\Phi'_H(\zeta)=\frac{\partial U_1}{\partial x}e_1+\frac{\partial U_2}{\partial x}e_2+\frac{\partial U_3}{\partial x}e_3+\frac{\partial U_4}{\partial x}e_4=\frac{\partial \Phi}{\partial x}.$$
Теорема доказана.
\end{proof}

\begin{theorem}\label{teor-pohidna-dobutku}
Если отображения $\Phi:\Omega_\zeta\rightarrow\mathbb{H(C)}$ и $\Psi:\Omega_\zeta\rightarrow\mathbb{H(C)}$ являются $H$-моногенными в области $\Omega_\zeta$, то произведение $\Phi\cdot\Psi$ также является $H$-моногенным отображением в $\Omega_\zeta$, при этом
$$d(\Phi\cdot\Psi)=d\Phi\cdot\Psi+\Phi\cdot d\Psi.$$
\end{theorem}

\begin{proof}
 Пусть
 $$\Phi(\zeta)=\sum\limits_{k=1}^4U_k(x,y,z)e_k\,,\quad \Psi(\zeta)=\sum\limits_{k=1}^4V_k(x,y,z)e_k\,.$$
Тогда
$$d\Phi=\sum\limits_{k=1}^4\left(\frac{\partial U_k}{\partial x}dx+\frac{\partial U_k}{\partial y}dy+\frac{\partial U_k}{\partial z}dz\right)e_k\,,$$
$$d\Psi=\sum\limits_{k=1}^4\left(\frac{\partial V_k}{\partial x}dx+\frac{\partial V_k}{\partial y}dy+\frac{\partial V_k}{\partial z}dz\right)e_k\,$$
и
$$d(\Phi\cdot\Psi)=d\Big(U_1V_1+U_3V_4\Big)e_1+d\Big(U_2V_2+U_4V_3\Big)e_2+$$
$$+d\Big(U_1V_3+U_3V_2\Big)e_3+d\Big(U_2V_4+U_4V_1\Big)e_4=$$
$$=\Bigg[\left(\frac{\partial U_1}{\partial x}V_1+\frac{\partial V_1}{\partial x}U_1+\frac{\partial U_3}{\partial x}V_4+\frac{\partial V_4}{\partial x}U_3\right)dx+$$
$$+\left(\frac{\partial U_1}{\partial y}V_1+\frac{\partial V_1}{\partial y}U_1+\frac{\partial U_3}{\partial y}V_4+\frac{\partial V_4}{\partial y}U_3\right)dy+$$
$$+\left(\frac{\partial U_1}{\partial z}V_1+\frac{\partial V_1}{\partial z}U_1+\frac{\partial U_3}{\partial z}V_4+\frac{\partial V_4}{\partial z}U_3\right)dz\Bigg]e_1+$$

$$+\Bigg[\left(\frac{\partial U_2}{\partial x}V_2+\frac{\partial V_2}{\partial x}U_2+\frac{\partial U_4}{\partial x}V_3+\frac{\partial V_3}{\partial x}U_4\right)dx+$$
$$+\left(\frac{\partial U_2}{\partial y}V_2+\frac{\partial V_2}{\partial y}U_2+\frac{\partial U_4}{\partial y}V_3+\frac{\partial V_3}{\partial y}U_4\right)dy+$$
$$+\left(\frac{\partial U_2}{\partial z}V_2+\frac{\partial V_2}{\partial z}U_2+\frac{\partial U_4}{\partial z}V_3+\frac{\partial V_3}{\partial z}U_4\right)dz\Bigg]e_2+$$

$$+\Bigg[\left(\frac{\partial U_1}{\partial x}V_3+\frac{\partial V_3}{\partial x}U_1+\frac{\partial U_3}{\partial x}V_2+\frac{\partial V_2}{\partial x}U_3\right)dx+$$
$$+\left(\frac{\partial U_1}{\partial y}V_3+\frac{\partial V_3}{\partial y}U_1+\frac{\partial U_3}{\partial y}V_2+\frac{\partial V_2}{\partial y}U_3\right)dy+$$
$$+\left(\frac{\partial U_1}{\partial z}V_3+\frac{\partial V_3}{\partial z}U_1+\frac{\partial U_3}{\partial z}V_2+\frac{\partial V_2}{\partial z}U_3\right)dz\Bigg]e_3+$$

$$+\Bigg[\left(\frac{\partial U_2}{\partial x}V_4+\frac{\partial V_4}{\partial x}U_2+\frac{\partial U_4}{\partial x}V_1+\frac{\partial V_1}{\partial x}U_4\right)dx+$$
$$+\left(\frac{\partial U_2}{\partial y}V_4+\frac{\partial V_4}{\partial y}U_2+\frac{\partial U_4}{\partial y}V_1+\frac{\partial V_1}{\partial y}U_4\right)dy+$$
$$+\left(\frac{\partial U_2}{\partial z}V_4+\frac{\partial V_4}{\partial z}U_2+\frac{\partial U_4}{\partial z}V_1+\frac{\partial V_1}{\partial z}U_4\right)dz\Bigg]e_4.$$
Преобразуем полученное выражение к следующему виду:

\small

$$\Bigg(V_1\frac{\partial U_1}{\partial x}dx+V_1\frac{\partial U_1}{\partial y}dy+V_1\frac{\partial U_1}{\partial z}dz+
V_4\frac{\partial U_3}{\partial x}dx+V_4\frac{\partial U_3}{\partial y}dy+V_4\frac{\partial U_3}{\partial z}dz\Bigg)e_1+$$

$$+\Bigg(V_2\frac{\partial U_2}{\partial x}dx+V_2\frac{\partial U_2}{\partial y}dy+V_2\frac{\partial U_2}{\partial z}dz+
V_3\frac{\partial U_4}{\partial x}dx+V_3\frac{\partial U_4}{\partial y}dy+V_3\frac{\partial U_4}{\partial z}dz\Bigg)e_2+$$

$$+\Bigg(V_3\frac{\partial U_1}{\partial x}dx+V_3\frac{\partial U_1}{\partial y}dy+V_3\frac{\partial U_1}{\partial z}dz+
V_2\frac{\partial U_3}{\partial x}dx+V_2\frac{\partial U_3}{\partial y}dy+V_2\frac{\partial U_3}{\partial z}dz\Bigg)e_3+$$

$$+\Bigg(V_4\frac{\partial U_2}{\partial x}dx+V_4\frac{\partial U_2}{\partial y}dy+V_4\frac{\partial U_2}{\partial z}dz+
V_1\frac{\partial U_4}{\partial x}dx+V_1\frac{\partial U_4}{\partial y}dy+V_1\frac{\partial U_4}{\partial z}dz\Bigg)e_4+$$

$$+\Bigg(U_1\frac{\partial V_1}{\partial x}dx+U_1\frac{\partial V_1}{\partial y}dy+U_1\frac{\partial V_1}{\partial z}dz+
U_4\frac{\partial V_3}{\partial x}dx+U_4\frac{\partial V_3}{\partial y}dy+U_4\frac{\partial V_3}{\partial z}dz\Bigg)e_1+$$

$$+\Bigg(U_2\frac{\partial V_2}{\partial x}dx+U_2\frac{\partial V_2}{\partial y}dy+U_2\frac{\partial V_2}{\partial z}dz+
U_3\frac{\partial V_4}{\partial x}dx+U_3\frac{\partial V_4}{\partial y}dy+U_3\frac{\partial V_4}{\partial z}dz\Bigg)e_2+$$

$$+\Bigg(U_3\frac{\partial V_1}{\partial x}dx+U_3\frac{\partial V_1}{\partial y}dy+U_3\frac{\partial V_1}{\partial z}dz+
U_2\frac{\partial V_3}{\partial x}dx+U_2\frac{\partial V_3}{\partial y}dy+U_2\frac{\partial V_3}{\partial z}dz\Bigg)e_3+$$

$$+\Bigg(U_4\frac{\partial V_2}{\partial x}dx+U_4\frac{\partial V_2}{\partial y}dy+U_4\frac{\partial V_2}{\partial z}dz+
U_1\frac{\partial V_4}{\partial x}dx+U_1\frac{\partial V_4}{\partial y}dy+U_1\frac{\partial V_4}{\partial z}dz\Bigg)e_4\,,$$
\normalsize
откуда будем иметь
\small
$$\Big(V_1dU_1+V_4dU_3\Big)e_1+\Big(V_2dU_2+V_3dU_4\Big)e_2+\Big(V_3dU_1+V_2dU_3\Big)e_3+$$
$$+\Big(V_4dU_2+V_1dU_4\Big)e_4+\Big(U_1dV_1+U_3dV_4\Big)e_1+\Big(U_2dV_2+U_4dV_3\Big)e_2+$$
$$+\Big(U_1dV_3+U_3dV_2\Big)e_3+\Big(U_2dV_4+U_4dV_1\Big)e_4=d\Phi\cdot\Psi+\Phi\cdot d\Psi.$$
\normalsize
Теорема доказана.
\end{proof}

В силу теоремы \ref{teor-pohidna-dobutku} множество $H$-моногенных отображений со значениями в алгебре $\mathbb{H(C)}$ образует функциональную алгебру, поскольку произведение двух $H$-моногенных отображений также является $H$-моногенным отображением.

В следующей теореме устанавливается связь между $G$-моногенными и $H$-моногенными отображениями.

\begin{theorem}\label{teor-1}
Каждое право-$G$-моногенное отображение $\Phi:\Omega_\zeta\rightarrow\mathbb{H(C)}$
и каждое лево-$G$-моногенное отображение $\widehat{\Phi}:\Omega_\zeta\rightarrow\mathbb{H(C)}$
  в области $\Omega_\zeta$  являются $H$-моногенными отображениями в этой области.
\end{theorem}

\begin{proof}
Пусть $\Phi:\Omega_\zeta\rightarrow\mathbb{H(C)}$ --- право-$G$-моногенное отображение. Тогда существование частных производных первого порядка от компонент отображения $\Phi$ вытекает из существования производной Гато (равенство (\ref{ozn-r-G-monog})).
Покажем теперь, что дифференциал
\begin{equation}\label{-riv-1}
d\Phi=\frac{\partial\Phi}{\partial x}dx+\frac{\partial\Phi}{\partial y}dy+\frac{\partial\Phi}{\partial z}dz
\end{equation}
представим в виде (\ref{--riv-1}).

С этой целью заметим, что следствием равенства (\ref{-riv-1}) и условий (\ref{umova-r-K-R}) является равенство
$$d\Phi=\big(dx+i_2dy+i_3dz\big)\frac{\partial\Phi}{\partial x}=d\zeta\,\Phi'(\zeta),$$
т.~е. представление вида  (\ref{--riv-1}), в котором $A_1=1, B_1=\Phi'(\zeta)$.

Аналогично устанавливается, что следствием равенства (\ref{-riv-1}) при $\Phi=\widehat{\Phi}$ и условий (\ref{umova-l-K-R}) является равенство
$$d\widehat{\Phi}=\widehat{\Phi}'(\zeta)d\zeta,$$
т.~е. снова представление вида  (\ref{--riv-1}), в котором $A_1=\widehat{\Phi}'(\zeta), B_1=1$.
 Теорема доказана.
\end{proof}

Поскольку право- и лево-$G$-моногенные отображения являются $H$-моногенными, то их произведения также являются $H$-моногенными отображениями. Поэтому следствием теорем \ref{teor-pohidna-dobutku}, \ref{teor-1} и представлений  (\ref{Phi-r-rozklad}),  (\ref{Phi-l-rozklad}) является следующее утверждение.

\begin{corollary}
 Пусть область $\Omega\subset \mathbb{R}^{3}$ является выпуклой в направлении прямых $L^1$, $L^2$ и $f_1(E_3)=f_2(E_3)=\mathbb{C}$. Тогда $H$-моногенными в области $\Omega_\zeta$ являются отображения
$$\Phi(\zeta)\cdot\widehat{\Phi}(\zeta)=$$
$$=\Big(F_1(\xi_1)\widehat{F}_1(\xi_1)+F_3(\xi_1)\widehat{F}_4(\xi_1)\Big)e_1+\Big(F_2(\xi_2)\widehat{F}_2(\xi_2)+F_4(\xi_2)\widehat{F}_3(\xi_2)\Big)e_2+$$
$$+\Big(F_1(\xi_1)\widehat{F}_3(\xi_2)+F_3(\xi_1)\widehat{F}_2(\xi_2)\Big)e_3+\Big(F_2(\xi_2)\widehat{F}_4(\xi_1)+F_4(\xi_2)\widehat{F}_1(\xi_1)\Big)e_4,$$
\vskip 1mm
$$\widehat{\Phi}(\zeta)\cdot\Phi(\zeta)=$$
$$=\Big(\widehat{F}_1(\xi_1)F_1(\xi_1)+\widehat{F}_3(\xi_2)F_4(\xi_2)\Big)e_1+\Big(\widehat{F}_2(\xi_2)F_2(\xi_2)+\widehat{F}_4(\xi_1)F_3(\xi_1)\Big)e_2+$$
$$+\Big(\widehat{F}_1(\xi_1)F_3(\xi_1)+\widehat{F}_3(\xi_2)F_2(\xi_2)\Big)e_3+\Big(\widehat{F}_2(\xi_2)F_4(\xi_2)+\widehat{F}_4(\xi_1)F_1(\xi_1)\Big)e_4,$$
где аналитические функции $F_k,\widehat{F}_k$ определены в равенствах \em(\ref{Phi-r-rozklad})\em,  \em(\ref{Phi-l-rozklad})\em.
\end{corollary}

В то же время существуют $H$-моногенные отображения, не являющиеся ни право-$G$-моногенными, ни лево-$G$-моногенными.

\begin{example}\label{example}
Отображение
$$h(\zeta)=(e^{\xi_1}+\xi_2^2)\,e_1+\xi_1\sin\xi_2\,e_2+\xi_2^2\,e_3+e^{\xi_1}\,e_4$$ является $H$-моногенным в пространстве $E_3$, но не является ни лево-$G$-моногенным, ни право-$G$-моногенным. Действительно, дифференциал этого отображения представляется в виде (\ref{--riv-1}):
$$dh=e^{\xi_1}e_1d\zeta e_1+\xi_1\cos\xi_2\,e_2\,d\zeta e_2+2\xi_2\,e_3\,d\zeta e_2+$$
$$+e^{\xi_1}\,e_4\,d\zeta e_1+2\xi_2\,e_3\,d\zeta\, e_4+ \sin\xi_2\,e_4\,d\zeta\, e_3.$$
Однако отображение $h$ не представляется ни в виде (\ref{Phi-r-rozklad}), ни в виде (\ref{Phi-l-rozklad}).
\end{example}

$H$-моногенное отображение $\Phi$, дифференциал которого представляется в виде
\begin{equation}\label{r-H-dPhi}
d\Phi=d\zeta\,\Phi'_H(\zeta)
\end{equation}
 будем называть \emph{право-$H$-моногенным}, а $H$-моногенное отображение $\widehat{\Phi}$, дифференциал которого представляется в виде
\begin{equation}\label{l-H-dPhi}
d\widehat{\Phi}=\widehat{\Phi}'_H(\zeta)d\zeta
\end{equation}
  --- \emph{лево-$H$-моногенным} в области $\Omega_\zeta$.

Установим необходимые и достаточные условия $G$-моногенности отображения.

\begin{theorem}\label{teor-2} Пусть компоненты $U_k:\Omega\rightarrow\mathbb{C}$ отображения  \em (\ref{Phi-U}) \em являются $\mathbb{R}$-дифференцируемыми в области $\Omega$.
Отображение $\Phi:\Omega_\zeta\rightarrow\mathbb{H(C)}$
является право-$G$-моногенным тогда и только тогда, когда оно ---
право-$H$-моногенное, а отображение
$\widehat{\Phi}:\Omega_\zeta\rightarrow\mathbb{H(C)}$ является
 лево-$G$-моногенным  тогда и только тогда, когда оно --- лево-$H$-моногенное.
   \end{theorem}

\begin{proof} Необходимость доказана при доказательстве теоремы \ref{teor-1}. Докажем достаточность.
Пусть отображение $\Phi$ --- право-$H$-моногенное, т.~е. выполняется равенство
(\ref{r-H-dPhi}). Следствием равенств (\ref{-riv-1}) и (\ref{r-H-dPhi}) является равенство
$$\frac{\partial \Phi}{\partial x}dx+\frac{\partial \Phi}{\partial y}dy+\frac{\partial \Phi}{\partial z}dz=d\zeta\Phi'_H(\zeta).
$$
С учетом выражений $\Phi'_H(\zeta)=\frac{\partial \Phi}{\partial x}$ и $d\zeta=dx+i_2dy+i_3dz$ имеем тождество
$$\frac{\partial \Phi}{\partial x}dx+\frac{\partial \Phi}{\partial y}dy+\frac{\partial \Phi}{\partial z}dz=\frac{\partial \Phi}{\partial x}dx+i_2\frac{\partial \Phi}{\partial x}dy+i_3\frac{\partial \Phi}{\partial x}dz,$$
следствием которого являются условия Коши -- Римана (\ref{umova-r-K-R}). Тогда по теореме 1 из
\cite{Shpakivskiy-Kuzmenko} отображение $\Phi$ --- право-$G$-моногенное.

Аналогично рассматривается случай лево-$H$-моногенного отображения.
Теорема доказана.
\end{proof}

Из теоремы \ref{teor-2} и теоремы 5 из \cite{Shpakivskiy-Kuzmenko} вытекает

\begin{corollary}
Если область $\Omega\subset
\mathbb{R}^{3}$ является выпуклой в направлении прямых $L^1$, $L^2$ и $f_1(E_3)=f_2(E_3)=\mathbb{C}$, то
каждое право-$H$-моногенное отображение $\Phi:\Omega_\zeta\rightarrow\mathbb{H(C)}$
представляется в виде \em(\ref{Phi-r-rozklad})\em и
каждое лево-$H$-моногенное отображение $\widehat{\Phi}:\Omega_\zeta\rightarrow\mathbb{H(C)}$ представляется в виде \em(\ref{Phi-l-rozklad})\em.
\end{corollary}

Следующая теорема содержит критерии право-$G$-моногенности и лево-$G$-моногенности отображений.

\begin{theorem}\label{krit-G-monog}
Отображение $\Phi:\Omega_\zeta\rightarrow\mathbb{H(C)}$  \big(или $\widehat{\Phi}:\Omega_\zeta\rightarrow\mathbb{H(C)}$\big) является право-$G$-моногенным \big(или лево-$G$-моногенным\big) в области $\Omega_\zeta\subset E_3$ тогда и только тогда, когда выполняется одно из следующих условий:

\em(I)\em\, компоненты $U_k:\Omega\rightarrow\mathbb{C}$ разложения \em (\ref{Phi-U}) \em являются $\mathbb{R}$-дифференцируемыми в области $\Omega_\zeta$ и выполняются условия \em(\ref{umova-r-K-R}) \em \big(или  \em(\ref{umova-l-K-R})\em\big) в каждой точке области $\Omega_\zeta$;

\em(II)\em\, компоненты $U_k:\Omega\rightarrow\mathbb{C}$ разложения \em (\ref{Phi-U}) \em являются $\mathbb{R}$-дифференцируемыми в области $\Omega_\zeta$ и отображение $\Phi$ \big(или $\widehat{\Phi}$\big) --- право-$H$-моногенное \big(или лево-$H$-моногенное\big) в $\Omega_\zeta$.

Если $f_1(E_3)=f_2(E_3)=\mathbb{C}$, то отображение $\Phi$ является право-$G$-моногенным \big(или $\widehat{\Phi}$ --- лево-$G$-моногенным\big) тогда и только тогда, когда выполняется одно из условий:

\em(IIІ)\em\, для каждой точки $\zeta_0\in\Omega_\zeta$ найдется окрестность, в которой отображение $\Phi$  \big(или $\widehat{\Phi}$\big)  разлагается в степенной ряд  \em(\ref{r-Taylor}) \em \big(или  \em(\ref{l-Taylor})\em\big);

\em(ІV)\em\, отображение $\Phi$ \big(или $\widehat{\Phi}$\big) непрерывно и удовлетворяет равенству \em(\ref{int-morera-r})\em\, \big(или  \em(\ref{int-morera-l})\em\big) для каждого треугольника $\Delta_\zeta$ такого, что $\overline{\Delta_\zeta}\subset\Omega_\zeta$.

Если $f_1(E_3)=f_2(E_3)=\mathbb{C}$ и, кроме того, область $\Omega\subset
\mathbb{R}^{3}$ является выпуклой в направлениии прямых $L^1$, $L^2$, то отображение $\Phi$ \big(или $\widehat{\Phi}$\big) --- право-$G$-моногенное \big(или $\widehat{\Phi}$ --- лево-$G$-моногенное\big) тогда и только тогда, когда

\em(V)\em\, существуют единственные аналитические в области $D_1:=\{\xi_1=x+a_1y+b_1z:(x,y,z)\in\Omega\}$
функции $F_1,\,F_3$ \big(или $\widehat{F}_1,\,\widehat{F}_4$\big) и единственные
аналитические в области $D_2:=\{\xi_2=x+a_2y+b_2z:(x,y,z)\in\Omega\}$
функции $F_2,\,F_4$ \big(или $\widehat{F}_2,\,\widehat{F}_3$\big) такие, что в области $\Omega_\zeta$ отображение $\Phi$ \big(или $\widehat{\Phi}$\big) представляется в виде \em(\ref{Phi-r-rozklad})\,\,\em \big(или \em(\ref{Phi-l-rozklad})\em\big).
\end{theorem}

\begin{proof}
Эквивалентность условия (I) и свойства право-$G$-моногенности установлена в
теореме 1 из \cite{Shpakivskiy-Kuzmenko}. Эквивалентность условия (II) и право-$G$-моногенности установлена в теореме \ref{teor-2}. Эквивалентность условия (ІІІ) и право-$G$-моногенности  вытекает из теоремы 1 работы \cite{Kuzm-series} и свойства сходящегося ряда (\ref{r-Taylor}) определять функцию, право-$G$-моногенную в шаре сходимости. Эквивалентность условия (ІV) и право-$G$-моногенности вытекает из теоремы \ref{teor-Morera} и теоремы 2 работы \cite{Shp-Kuzm-intteor}.

Наконец, для доказательства эквивалентности условия (V) и право-$G$-моногенности  отображения $\Phi$ достаточно заметить, что отображение (\ref{Phi-r-rozklad}) является право-$G$-моногенным в $\Omega_\zeta$, а единственность функций $F_1,\,F_2,\,F_3,\,F_4$  из (\ref{Phi-r-rozklad}) следует из единственности разложения элемента алгебры $\mathbb{H(C)}$ по базису $\{e_1,e_2,e_3,e_4\}$. В случае лево-$G$-моногенного отображения теорема доказывается аналогично.
Теорема доказана.
 \end{proof}

\vskip 2mm
\textbf{Благодарности.} Авторы признательны профессору С.\,А.~Плаксе за ценные советы, которые способствовали улучшению работы.

\bigskip

СВЕДЕНИЯ ОБ АВТОРАХ

\medskip
Виталий Станиславович Шпаковский\\
Институт математики НАН Украины,\\
ул. Терещенковская, 3, Киев, Украина\\
shpakivskyi86@gmail.com

\medskip
Татьяна Сергеевна Кузьменко\\
Институт математики НАН Украины,\\
ул. Терещенковская, 3, Киев, Украина\\
kuzmenko.ts15@gmail.com
\end{document}